\documentclass{amsart}
\usepackage{amssymb,mathtools}
\usepackage[british]{babel}
\usepackage{enumitem}
\usepackage[latin1]{inputenc}

\usepackage{hyperref}

\newtheorem{theorem}{Theorem}
\numberwithin{theorem}{section}
\newtheorem{corollary}[theorem]{Corollary}

\newtheorem{proposition}[theorem]{Proposition}
\theoremstyle{definition}
\newtheorem{definition}[theorem]{Definition}
\newtheorem{remark}[theorem]{Remark}

\numberwithin{equation}{section}

\title[Fra\"iss\'e's conjecture and partial impredicativity]{Fra\"iss\'e's conjecture, partial impredicativity\\ and well-ordering principles, part~I}
\author{Anton Freund}

\address{University of W\"urzburg, Institute of Mathematics, Emil-Fischer-Str.~40, 97074 W\"urz\-burg, Germany}
\email{anton.freund@uni-wuerzburg.de}

\thanks{Funded by the Deutsche Forschungsgemeinschaft (DFG, German Research Foundation) -- Project number 460597863.}

\begin{document}

\begin{abstract}
Fra\"iss\'e's conjecture (proved by Laver) is implied by the~$\Pi^1_1$-com\-prehension axiom of reverse mathematics, as shown by Montalb\'an. The implication must be strict for reasons of quantifier complexity, but it seems that no better bound has been known. We locate such a bound in a hierarchy of Suzuki and Yokoyama, which extends Towsner's framework of partial impredicativity. Specifically, we show that Fra\"iss\'e's conjecture is implied by a principle of pseudo $\Pi^1_1$-comprehension. As part of the proof, we introduce a cofinite version of the $\Delta^0_2$-Ramsey theorem, which may be of independent interest. We also relate pseudo $\Pi^1_1$-comprehension to principles of pseudo $\beta$-model reflection (due to Suzuki and Yokoyama) and reflection for $\omega$-models of transfinite induction (studied by Rathjen and Valencia-Vizca\'ino). In a forthcoming companion paper, we characterize pseudo \mbox{$\Pi^1_1$-}compre\-hension by a well-ordering principle, to get a transparent combinatorial bound for the strength of Fra\"iss\'e's conjecture.
\end{abstract}


\maketitle

\section{Introduction}

By Fr\"aiss\'e's conjecture, we mean the statement that any infinite sequence of countable linear orders $L_0,L_1,\ldots$ admits $i<j$ such that $L_i$ embeds into~$L_j$. Laver has proved that this conjecture holds, even for $\sigma$-scattered rather than countable orders. When formalized in the framework of reverse mathematics~\cite{friedman-rm,simpson09}, Laver's proof uses the extremely strong axiom of $\Pi^1_2$-comprehension. By ground\-breaking work of Montalb\'an~\cite{montalban-fraisse}, the strong but much weaker axiom of $\Pi^1_1$-compre\-hension suffices for a proof of Fra\"iss\'e's conjecture. This upper bound must be strict for reasons of quantifier complexity. Specifically, Fra\"iss\'e's conjecture is a $\Pi^1_2$-statement, and no such statement can be equivalent to $\Pi^1_1$-comprehension (use Theorem~VII.2.10 of~\cite{simpson09} to show that $\Pi^1_1$-comprehension entails the consistency of its $\Pi^1_2$-consequences). The best lower bound from the literature, which is due to Shore~\cite{shore-comp-wos}, says that Fra\"iss\'e's conjecture implies arithmetical transfinite recursion (see~\cite{freund-manca-fraisse} for a small correction). It is not known whether the latter suffices to prove Fra\"iss\'e's conjecture.

There are several mathematical theorems of complexity $\Pi^1_2$ for which the only known or the most `natural' proof establishes $\Pi^1_1$-comprehension as an upper bound, which is necessarily suboptimal. To deal with this situation, Towsner~\cite{partial-impred} has introduced axiom systems for `partial impredicativity'. Inspired by the functional interpretation that is also used in proof mining~\cite{kohlenbach-book}, these replace $\Pi^1_1$-comprehension by certain $\Pi^1_2$-approximations, which support relatively straightforward modifications of the natural proofs. In particular, Towsner shows that well-known theorems of Menger and Nash-Williams can can be derived from his strongest axiom $\mathsf{TLPP}$ (`trans\-finite leftmost path principle'), while a weaker version of the axiom suffices for the elegant proof of \mbox{Kruskal's theorem} via (relatively) minimal bad sequences.

There have been important developments concerning partial impredicativity in recent years. In particular, Fern\'andez-Duque, Shafer, Towsner and Yokoyama~\cite{TLPP-Caristi} established the first reversal by showing that $\mathsf{TLPP}$ is equivalent to a version of Caristi's theorem. Suzuki and Yokoyama~\cite{suzuki-yokoyama} defined a hierarchy that starts with Towsner's systems and exhausts all \mbox{$\Pi^1_2$-}consequences of $\Pi^1_1$-comprehension. More precisely, $\mathsf{TLPP}$ lies between the first two stages of this hierarchy. The same holds for the pseudo $\Pi^1_1$-comprehension axioms from the following definition, as we show in Section~\ref{sect:pseudo-Pi11}. Indeed, pseudo $\Pi^1_1$-comprehension is essentially the same as a pseudo hyperjump principle and equivalent to a reflection principle for pseudo $\beta$-models, which were both introduced by Suzuki and Yokoyama~\cite{suzuki-yokoyama}. We shall also establish a connection with reflection for $\omega$-models of transfinite induction (also known as bar induction), as studied by Rathjen and Valencia-Vizca\'ino~\cite{rathjen-model-bi}. The following relies on standard notions from reverse mathematics, which are explained, e.\,g., in~\cite{simpson09}.

\begin{definition}\label{def:Pi11X}
For an $\omega$-model~$\mathcal M\ni X$, we write $\mathcal M\vDash\Pi^1_1\textsf{-CA}_0(X)$ to express that $\mathcal M$ contains $\{x\,|\,\mathcal M\vDash\varphi(x,X)\}$ for any $\Pi^1_1$-formula~$\varphi$ with no set parameters other than~$X$. The pseudo $\Pi^1_1$-comprehension axioms $\Pi^1_1\textsf{-CA}^\varepsilon$ and $\Pi^1_1\textsf{-CA}^\Gamma$ assert that any set~$X\subseteq\mathbb N$ admits an $\omega$-model~$\mathcal M\ni X$ with $\mathcal M\vDash\Pi^1_1\textsf{-CA}_0(X)$ and $\mathcal M\vDash\mathsf{ACA}_0$ or~$\mathcal M\vDash\mathsf{ATR}_0$, respectively. For $\star\in\{\varepsilon,\Gamma\}$, we let $\Pi^1_1\textsf{-CA}_0^\star$ denote $\mathsf{ACA}_0+\Pi^1_1\textsf{-CA}^\star$.
\end{definition}

Let us note that $\Pi^1_1\textsf{-CA}_0^\Gamma$ contains~$\mathsf{ATR}_0$ (since a $\Pi^1_2$-state\-ment is true when it holds in some~$\omega$-model from every cone). In fact, $\Pi^1_1\textsf{-CA}_0^\varepsilon$ contains~$\mathsf{ATR}_0$ as well, by the proof of Proposition~\ref{prop:Pi11eps-BI} below (or by a more direct argument). To explain our terminology and notation, we point out that Suzuki and Yokoyama have studied closely related principles that they call pseudo Ramsey theorems. The superscripts $\varepsilon$ and $\Gamma$ refer to the proof-theoretic ordinals of~$\mathsf{ACA}_0$ and $\mathsf{ATR}_0$.

Pseudo $\Pi^1_1$-comprehension may be seen as a relativization of the parameter-free $\Pi^1_1$-compre\-hension principle~$\Pi^1_1\textsf{-CA}^-$. We note that the proof-theoretic ordinal of the theory $\mathsf{ACA}_0+\Pi^1_1\textsf{-CA}^-$ is the Bachmann-Howard ordinal $\vartheta(\varepsilon_{\Omega+1})$, which can be described via a single collapsing function~$\vartheta$, while the proof-theoretic ordinal of the theory $\Pi^1_1\textsf{-CA}_0$ relies on a hierarchy of $\omega$-many collapsing functions (see~\cite{howard-ID1,takeuti67} and the presentation in~\cite{pohlers-proof-theory,pohlers98}). Thus the restriction on parameters weakens the principle considerably. For $\Pi^1_1\textsf{-CA}_0^\varepsilon$ and $\Pi^1_1\textsf{-CA}_0^\Gamma$ we get a similar picture from the hierarchy of Suzuki and Yokoyama and from the~equi\-valences with well-ordering principles that are explained below. In this sense, the following improves the bound from~\cite{montalban-fraisse} significantly.

\begin{theorem}\label{thm:Fraisse-Pi11star}
Fra\"iss\'e's conjecture is provable in~$\Pi^1_1\textsf{-}\mathsf{CA}_0^\Gamma$.
\end{theorem}

In Section~\ref{sect:cof-RT} we will actually prove a somewhat stronger result that involves a new cofinite version of the $\Delta^0_2$-Ramsey theorem (see Definition~\ref{def:cof-RT}). Our proof is a relatively straightforward modification of Montalb\'an's, though we will face one new challenge (see Remark~\ref{rmk:cof-RT-colours}).

A forthcoming companion paper will characterize $\Pi^1_1\textsf{-CA}^\Gamma$ by a well-ordering principle. Such a principle asserts that $F(X)$ is well-founded for any well order~$X$, where $F$ is a computable transformation of linear orders. The best-known example is probably the transformation of~$X$ into the linear order with underlying set
\begin{equation*}
\omega(X)=\left\{\left.\omega^{x(0)}+\ldots+\omega^{x(n-1)}\,\right|\,x(i)\in X\text{ and }x(n-1)\leq_X\ldots\leq_X x(0)\right\},
\end{equation*}
ordered by lexicographic comparisons of the exponents. Girard~\cite{girard87} and Hirst~\cite{hirst94} have shown that, over the base theory~$\mathsf{RCA}_0$, arithmetical comprehension is equivalent to the principle that $\omega(X)$ is well-founded for any well order~$X$. The literature now includes well-ordering principles that correspond to infinite iterations of the Turing jump~\cite{marcone-montalban}, arithmetical transfinite recursion \cite{rathjen-weiermann-atr} (originally an unpublished result of H.~Friedman), $\omega$-models of arithmetical transfinite recursion~\cite{rathjen-atr}, $\omega$-models of transfinite induction~\cite{rathjen-model-bi} and $\omega$-models of \mbox{$\Pi^1_1$-}comprehension without~\cite{thomson-rathjen-Pi-1-1} and with~\cite{thomson-thesis} transfinite induction. The principles of $\Pi^1_1$-comprehension and $\Pi^1_1$-transfinite recursion (which have complexity~$\Pi^1_3$) have been characterized by well-ordering principles of higher type~\cite{freund-equivalence,freund-computable,FR_Pi11-recursion}.

In the following section, we will see that $\Pi^1_1\textsf{-CA}^\varepsilon$ is equivalent to the principle that every set lies in an $\omega$-model of transfinite induction. This is one of the principles that were referenced in the previous paragraph. Rathjen and Valencia-Vizca\'ino~\cite{rathjen-model-bi} have shown that it is equivalent to the statement that a certain order~$\vartheta_X=\vartheta(\varepsilon_{\Omega+X})$ is well-founded for any well order~$X$. The order $\vartheta(\varepsilon_{\Omega+X})$ may be seen as a relativization of the Bachmann-Howard ordinal. It contains elements $\varepsilon_{\Omega+\alpha}$ for~$\alpha\in X$ that represent `large' $\varepsilon$-numbers (i.\,e., fixed points of ordinal exponentiation) and comes with a so-called collapsing function $\vartheta$ that maps $\vartheta(\varepsilon_{\Omega+X})$ into a proper initial segment of itself. This function cannot be an embedding, but it is `almost' order preserving, which forces~$\Omega$ to be large (see~\cite{rathjen-model-bi} for details).

The result of Rathjen and Valencia-Vizca\'ino~\cite{rathjen-model-bi} entails that $\Pi^1_1\textsf{-CA}^\varepsilon$ is equivalent to the principle that $X\mapsto\vartheta(\varepsilon_{\Omega+X})$ preserves well orders (see Corollary~\ref{cor:ALPP-WO-Princ}). As far as the author is aware, this is the first time an explicit connection is made between well-ordering principles and Towsner's partial impredicativity. In the aforementioned companion paper, we prove an analogous characterization of $\Pi^1_1\textsf{-CA}^\Gamma$ in terms of orders $\vartheta(\Gamma_{\Omega+X})$, which contain terms~$\Gamma_{\Omega+\alpha}$ that represent fixed points of the Veblen function. This characterization is of interest in connection with Fra\"iss\'e's conjecture because it provides a more combinatorial upper bound. At the same time, its proof involves a rather technical ordinal analysis, so that the companion paper addresses a somewhat different readership. Let us note that $\Pi^1_1\textsf{-CA}^\Gamma$ does not seem to have a characterization via $\omega$-models of transfinite induction (see the paragraph after the proof of Proposition~\ref{prop:Pi11eps-BI}). For this reason, we need a different approach than in~\cite{rathjen-model-bi} to characterize $\Pi^1_1\textsf{-CA}^\Gamma$ by a well-ordering principle.

\subsection*{Acknowledgements} I am very grateful to Davide Manca, Yudai Suzuki and Keita Yokoyama for their helpful feedback on a first version of this paper.

\section{Pseudo $\Pi^1_1$-comprehension}\label{sect:pseudo-Pi11}

In this section, we relate pseudo $\Pi^1_1$-comprehension to the pseudo \mbox{$\beta$-}model reflection of Suzuki and Yokoyama~\cite{suzuki-yokoyama},  to the leftmost path principles of Towsner~\cite{partial-impred} and to reflection for $\omega$-models of transfinite induction, which Rathjen and Valencia-Vizca\'ino~\cite{rathjen-model-bi} have studied in connection with well-ordering principles.

To explain the original approach of Towsner~\cite{partial-impred}, we first recall that $\Pi^1_1$-compre\-hension is equivalent to the statement that any ill-founded tree has a left-most branch~\cite{marcone-bad-sequence}. The arithmetical leftmost path principle~$\mathsf{ALPP}$ asserts that any ill-founded tree~$T$ has a branch~$f$ such that no branch of~$T$ is to the left of~$f$ and arithmetically reducible to~$T\oplus f$ (see Section~4 of~\cite{suzuki-yokoyama} for more details). For the transfinite leftmost path principle~$\mathsf{TLPP}$, one admits all paths that are $\Sigma_\alpha$ in~$T\oplus f$ for some well order~$\alpha$.

Suzuki and Yokoyama~\cite{suzuki-yokoyama} base their approach on principles $\beta^1_0\mathsf{RFN}(n;\varphi)$, which assert that any $X$ admits coded $\omega$-models~$\mathcal M_0\in\ldots\in\mathcal M_n$ with $X\in\mathcal M_0$ and $\mathcal M_n\vDash\mathsf{ACA}_0+\varphi$ as well as $\mathcal M_{i+1}\vDash\text{``$\mathcal M_i$ is a $\beta$-model"}$ for each~$i<n$. They show that any $\Pi^1_2$-theorem of $\Pi^1_1\textsf{-CA}_0$ is provable in~$\mathsf{ACA}_0+\beta^1_0\mathsf{RFN}(n;\top)$ for some~$n\in\mathbb N$. Furthermore, they show that we have
\begin{equation*}
\mathsf{ALPP}=\beta^1_0\mathsf{RFN}(1;\top)<\mathsf{TLPP}<\beta^1_0\mathsf{RFN}(1;\mathsf{ATR})<\beta^1_0\mathsf{RFN}(2;\top),
\end{equation*}
where $\varphi=\psi$ and $\varphi<\psi$ denote $\mathsf{ACA}_0\vdash\varphi\leftrightarrow\psi$ and $\mathsf{ACA}_0+\psi\vdash\mathsf{Con}(\mathsf{ACA}_0+\varphi)$. In the same notation, the following shows $\Pi^1_1\textsf{-CA}^\Gamma=\beta^1_0\mathsf{RFN}(1;\mathsf{ATR})$. The corollary below formulates an analogous result for $\Pi^1_1\textsf{-CA}^\varepsilon$. We note that pseudo $\Pi^1_1$-compre\-hension essentially coincides with the hyperjump principle from Lemma~3.8 of~\cite{suzuki-yokoyama}, so that the following can be seen as a special case of this lemma.

\begin{proposition}[$\mathsf{ACA}_0$]\label{prop:Pi11min-beta}
The principle $\Pi^1_1\text{-}\mathsf{CA}^\Gamma$ holds precisely if any~$X$ admits coded $\omega$-models~$\mathcal M_0\in\mathcal M_1$ with~$\mathcal M_1\vDash\mathsf{ATR}_0+\text{``$\mathcal M_0\ni X$ is a coded $\beta$-model"}$.
\end{proposition}
\begin{proof}
The result is a straightforward variant of the classical equivalence between $\Pi^1_1$-comprehension and the statement that any set is contained in a coded~$\beta$-model. To make this explicit, we first assume $\Pi^1_1\text{-}\mathsf{CA}^\Gamma$. Any given~$X$ is then contained in an $\omega$-model $\mathcal M\vDash\mathsf{ATR}_0+\Pi^1_1\textsf{-CA}(X)$. As the hyperjump of $X$ is $\Sigma^1_1$-definable without further set parameters (see Definition~VII.1.5 of~\cite{simpson09}), we get
\begin{equation*}
\mathcal M\vDash\text{``the hyperjump of~$X$ exists"}.
\end{equation*}
Due to Lemma~VII.2.9 of~\cite{simpson09}, this entails
\begin{equation*}
\mathcal M\vDash\text{``there is a coded $\beta$-model that contains~$X$"}.
\end{equation*}
We can take~$\mathcal M_1:=\mathcal M$ to validate the statement from the proposition. One may also omit the detour via hyperjumps and construct the desired $\beta$-model directly by $\Pi^1_1$-comprehension with $X$ as the only set parameter, following the proof of the cited Lemma~VII.2.9. For the converse direction, one can argue similarly via Lemmas~VII.1.6 and~VII.1.9 of~\cite{simpson09}, though a direct argument is much simpler here. Indeed, for $X\in\mathcal M_0\in\mathcal M_1$ as in the proposition, we get
\begin{equation*}
\{x\,|\,\mathcal M_1\vDash\varphi(x,X)\}=\{x\,|\,\mathcal M_0\vDash\varphi(x,X)\}\in\mathcal M_1
\end{equation*}
when~$\varphi$ is~$\Pi^1_1$, so that $\mathcal M_1$ has the property that is required by~$\Pi^1_1\text{-}\mathsf{CA}^\Gamma$.
\end{proof}

We also record the following version of the result.

\begin{corollary}[$\mathsf{ACA}_0$]\label{cor:alpp}
The principles $\Pi^1_1\text{-}\mathsf{CA}^\varepsilon$ and $\mathsf{ALPP}$ are equivalent.
\end{corollary}
\begin{proof}
We take up the notation from the paragraph before Proposition~\ref{prop:Pi11min-beta}. By~the proof of the latter (with $\mathsf{ACA}_0$ at the place of~$\mathsf{ATR}_0$ in each of the equivalent~statements), our principle $\Pi^1_1\text{-}\mathsf{CA}^\varepsilon$ is equivalent to~$\beta^1_0\mathsf{RFN}(1;\top)$. The latter is equivalent to~$\mathsf{ALPP}$ by Theorem~4.14 of~\cite{suzuki-yokoyama}.
\end{proof}

Towsner~\cite{partial-impred} has established several results that connect his relative leftmost path principles and transfinite induction. These do not seem to entail the following, which was independently proved by Yudai Suzuki (personal communication).

\begin{proposition}[$\mathsf{ACA}_0$]\label{prop:Pi11eps-BI}
The principle $\Pi^1_1\text{-}\mathsf{CA}^\varepsilon$ holds precisely if every set is contained in a coded $\omega$-model of $\mathsf{ACA}_0$ that satisfies the principle~$\Pi^1_\infty\text{-}\mathsf{TI}$ of transfinite induction for all formulas of second-order arithmetic.
\end{proposition}
\begin{proof}
Given $\Pi^1_1\text{-}\mathsf{CA}^\varepsilon$, any $X$ admits $\omega$-models~$\mathcal M_0\in\mathcal M_1$ with
\begin{equation*}
\mathcal M_1\vDash\mathsf{ACA}_0+\text{``$\mathcal M_0\ni X$ is a coded $\beta$-model"},
\end{equation*}
by the proof of Proposition~\ref{prop:Pi11min-beta}. For each~$n\in\mathbb N$, Lemma~VII.2.15 of~\cite{simpson09} yields
\begin{equation*}
\mathcal M_1\vDash\text{``$\Pi^1_n\text{-}\mathsf{TI}$ holds in~$\mathcal M_0$"}.
\end{equation*}
Now the statement $\mathcal M_0\vDash\Pi^1_n\text{-}\mathsf{TI}$ has complexity~$\Sigma^1_1$ (in fact~$\Delta^1_1$), as it asserts that there is a valuation that assigns certain truth values (cf.~Definition~VII.2.1 of~\cite{simpson09}). Since $\Sigma^1_1$-statements are upwards absolute for~$\omega$-models, we learn that $\mathcal M_0\vDash\Pi^1_n\text{-}\mathsf{TI}$ holds `in the real world', as required.

For the converse direction, assume any set is contained in an $\omega$-model of~$\Pi^1_\infty\text{-}\mathsf{TI}$. The latter implies~$\mathsf{ATR}$ (see Corollary~VII.2.19 of~\cite{simpson09}), which is thus available `in the real world' (cf.~the paragraph after Definition~\ref{def:Pi11X}). Given any~$X$, we now consider an $\omega$-model~$\mathcal M\ni X$ that satisfies $\mathsf{ACA}_0+\Pi^1_\infty\text{-}\mathsf{TI}$. Using $\mathsf{ATR}$ (in fact just arithmetical recursion along~$\mathbb N$), we can form the $\omega$-model~$\mathcal N$ that consists of those sets that are definable in~$\mathcal M$ (by any second-order formula with parameters). From the proof of Lemma~VII.2.17 in~\cite{simpson09}, we know that $\mathcal M$ is a $\beta$-submodel of $\mathcal N$ (the point being that $\mathcal M$ and $\mathcal N$ agree on well-foundedness due to~$\mathcal M\vDash\Pi^1_\infty\text{-}\mathsf{TI}$). When~$\varphi$ is $\Pi^1_1$, we thus get
\begin{equation*}
\{x\,|\,\mathcal N\vDash\varphi(x,X)\}=\{x\,|\,\mathcal M\vDash\varphi(x,X)\}\in\mathcal N,
\end{equation*}
as needed to conclude $\mathcal N\vDash\Pi^1_1\textsf{-CA}(X)$.
\end{proof}

Proposition~\ref{prop:Pi11min-beta} remains valid with essentially the same proof when~$\mathsf{ATR}$ is replaced by $\mathsf{ACA}$ on both sides of the equivalence (cf.~the proof of Corollary~\ref{cor:alpp}). In contrast, there seems to be no simple way to modify Proposition~\ref{prop:Pi11eps-BI} in order to characterize $\Pi^1_1\textsf{-CA}^\Gamma$ via $\omega$-models of transfinite induction. This is because these $\omega$-models already satisfy~$\mathsf{ATR}$, so that adding the latter to the base theory has no effect. In the previous proof, the model~$\mathcal N$ does not generally validate~$\mathsf{ATR}$.

As mentioned in the introduction, Rathjen and Valencia-Vizca\'ino~\cite{rathjen-model-bi} have related $\omega$-models of transfinite induction to a well-ordering principle~$X\mapsto\vartheta_X$, where the order $\vartheta_X=\vartheta(\varepsilon_{\omega+X})$ can be seen as a relativization of the Bachmann-Howard ordinal. If we combine their result with Corollary~\ref{cor:alpp} and Proposition~\ref{prop:Pi11eps-BI}, we obtain the following.

\begin{corollary}[$\mathsf{ACA}_0$]\label{cor:ALPP-WO-Princ}
The arithmetical leftmost path principle $\mathsf{ALPP}$ holds precisely if the order~$\vartheta_X$ from~\cite{rathjen-model-bi} is well-founded for every well order~$X$.
\end{corollary}

In a companion paper, we will characterize $\Pi^1_1\textsf{-CA}^\Gamma$ in terms of a well-ordering principle~$X\mapsto\vartheta(\Gamma_{\Omega+X})$. As we have not been able to express $\Pi^1_1\textsf{-CA}^\Gamma$ via $\omega$-models of trans\-finite induction, this requires a different approach than in~\cite{rathjen-model-bi}.

\section{The cofinite $\Delta^0_2$-Ramsey theorem}\label{sect:cof-RT}

In this section, we show that Fra\"iss\'e's conjecture is provable in~$\Pi^1_1\textsf{-CA}^\Gamma_0$. We first introduce a new cofinite version of the $\Delta^0_2$-Ramsey theorem. It will be shown that this version is provable in~$\Pi^1_1\textsf{-CA}^\Gamma_0$ and entails that the antichain with three elements is $\Delta^0_2\textsf{-bqo}$ (see the explanation below). The latter implies Fra\"iss\'e's conjecture by work of Montalb\'an~\cite{montalban-fraisse}.

Let us discuss some notation and terminology. The collections of finite and infinite subsets of~$X$ will be denoted by $[X]^{<\omega}$ and $[X]^\omega$, respectively. When we have~$X\subseteq\mathbb N$, we identify these subsets with their increasing enumerations. For a set $Y=\{Y(0),Y(1),\ldots\}\in[\mathbb N]^\omega$ with $Y(0)<Y(1)<\ldots$ we put~$Y[n]=\{Y(i)\,|\,i<n\}$. In the following definition, $Q$ can be any set, though we will be most interested in the case where it comes with a quasi-ordering.

\begin{definition}
We say that $f:[\mathbb N]^{<\omega}\to Q$ is eventually constant if each~$X\in[\mathbb N]^\omega$ admits an~$N\in\mathbb N$ with $f(X[n])=f(X[N])$ for all~$n>N$. When this is the case, we define~$\overline f:[\mathbb N]^\omega\to Q$ by stipulating that $\overline f(X)=f(X[n])$ holds for large~$n$.
\end{definition}

From the viewpoint of reverse mathematics, we point out that $\overline f(X)=q$ is a $\Delta^0_2$-relation. Conversely, any $\Delta^0_2$-definable function~$F:[\mathbb N]^\omega\to Q$ can be written as $F=\overline f$ for some eventually constant~$f$, provably in~$\mathsf{ACA}_0$ (see~\cite[Lemma~3.1]{montalban-fraisse}).

\begin{definition}\label{def:Del02-bqo}
A quasi order~$(Q,\leq_Q)$ is called $\Delta^0_2\textsf{-bqo}$ if each eventually constant function~$f:[\mathbb N]^{<\omega}\to Q$ admits an~$X\in[\mathbb N]^\omega$ with $\overline f(X)\leq_Q\overline f(X\backslash\{\min(X)\})$.
\end{definition}

The notion of $\mathsf{bqo}$ (better-quasi-order) is due to Nash-Williams~\cite{nash-williams-trees,nash-williams-bqo}. His~original definition essentially corresponds to the case where~$\overline f$ is \mbox{$\Delta^0_1$-}definable and hence continuous, in the sense that each~$X\in[\mathbb N]^\omega$ admits an~$N\in\mathbb N$ such that $X[N]=Y[N]$ implies~$\overline f(X)=\overline f(Y)$. Simpson~\cite{simpson-borel-bqos} has shown that one obtains an equivalent definition if one admits all Borel functions at the place of~$\overline f$. The equivalence relies on a rather strong metatheory. The case of $\Delta^0_2$-functions plays a crucial role in Montalb\'an's analysis of Fra\"iss\'e's conjecture.

\begin{theorem}[$\mathsf{ATR}_0$; \cite{montalban-fraisse}]\label{thm:montalban}
If the antichain with three elements is $\Delta^0_2\textsf{-bqo}$, then Fra\"iss\'e's conjecture holds.
\end{theorem}

Let us note that the antichain with two elements is provably $\Delta^0_2\textsf{-bqo}$ in weak theories~\cite{marcone-survey-old,montalban-fraisse}, while all finite orders reduce to the case of three elements~\cite{freund-3-bqo,provable-bqos}.

To conclude that Fra\"iss\'e's conjecture is provable in~$\Pi^1_1\textsf{-CA}_0$, Montalb\'an used \mbox{$\Pi^1_1$-}compre\-hension in the form of the $\Delta^0_2$-Ramsey theorem. As Fra\"iss\'e's conjecture is a $\Pi^1_2$-statement, we should only need a $\Pi^1_2$-approximation to the $\Delta^0_2$-Ramsey theorem in order to carry out the proof.

\begin{definition}\label{def:cof-RT}
By the cofinite $\Delta^0_2$-Ramsey theorem (abbreviated $\Delta^0_2\textsf{-RT(cof)}$), we mean the statement that each eventually constant function $f:[\mathbb N]^{<\omega}\to\{0,1\}$ admits an $X\in[\mathbb N]^\omega$ such that $\overline f$ is constant on~$\{Y\in[X]^\omega\,:\,X\backslash Y\text{ is finite}\}$.
\end{definition}

As revealed by the following proof, the restriction to cofinite subsets does not make a difference in the $\Delta^0_1$-case, i.\,e., for the clopen Ramsey theorem.

\begin{proposition}[$\mathsf{RCA}_0$]
The cofinite $\Delta^0_2$-Ramsey theorem entails arithmetical transfinite recursion.
\end{proposition}
\begin{proof}
We show that $\Delta^0_2\textsf{-RT(cof)}$ implies the clopen Ramsey theorem, which is known to imply~$\mathsf{ATR}$ (see Theorem~V.9.7 of~\cite{simpson09}). Consider an $\overline f:[\mathbb N]^\omega\to\{0,1\}$ that is continuous (cf.~the paragraph after Definition~\ref{def:Del02-bqo}). From $\Delta^0_2\textsf{-RT(cof)}$ we get $X\in[\mathbb N]^\omega$ and~$i<2$ such that $\overline f(Y)=i$ holds for any cofinite~$Y\subseteq X$. Now consider an arbitrary $Z\in[X]^\omega$. Given that $\overline f$ is continuous, we find an~$N\in\mathbb N$ such that $Z[N]=Z'[N]$ implies~$\overline f(Z)=\overline f(Z')$. Clearly, there is a cofinite~$Z'\subseteq X$ with~$Z[N]=Z'[N]$. We get $\overline f(Z)=i$, as the clopen Ramsey theorem demands.
\end{proof}

We note that $\Delta^0_2\textsf{-RT(cof)}$ is entailed by the principle $\mathsf{rel}(\Sigma^0_2\mathsf{Ram})$ of Suzuki and Yokoyama (though the present author came up with the cofinite $\Delta^0_2$-Ramsey theorem independently). It is shown in~\cite{suzuki-yokoyama} that $\mathsf{rel}(\Sigma^0_2\mathsf{Ram})$ can be derived from the principle~$\beta^1_0\mathsf{RFN}(2)$ that was mentioned in the previous section. The following is a (straight\-forward) improvement of this bound.

\begin{proposition}\label{prop:D02Ram-Pi11star}
The cofinite $\Delta^0_2$-Ramsey theorem is provable in~$\Pi^1_1\text{-}\mathsf{CA}^\Gamma_0$.
\end{proposition}
\begin{proof}
Aiming at $\Delta^0_2\textsf{-RT(cof)}$, we consider an arbitrary $f:[\mathbb N]^{<\omega}\to 2$ that is eventually constant. We invoke Proposition~\ref{prop:Pi11min-beta} to get coded $\omega$-models~$\mathcal M_0\in\mathcal M_1$~with
\begin{equation*}
\mathcal M_1\vDash\mathsf{ATR}_0+\text{``$\mathcal M_0\ni f$ is a coded $\beta$-model"}.
\end{equation*}
Note that $\mathcal M_1$ satisfies the statement that $f$ is eventually constant, as the latter has complexity~$\Pi^1_1$. Hence $\overline f$ is correctly defined in~$\mathcal M_1$. It suffices to show that there is an infinite~$X\in\mathcal M_1$ with
\begin{equation*}
\mathcal M_1\vDash\text{``$\overline f$ is constant on~$[X]^\omega$"},
\end{equation*}
since any cofinite subset of~$X$ is an element of~$\mathcal M_1$. The open claim holds essentially by Lemma~VI.6.2 of~\cite{simpson09}, which proves the arithmetical Ramsey theorem with the help of $\beta$-models. In order to make this more explicit, we write $\exists x\,\psi(x,Y)$ for the $\Sigma^0_2$-formula~$\overline f(Y)=0$ with parameter~$f$ (where $\psi$ is $\Pi^0_1$). Recall that infinite subsets of~$\mathbb N$ can be identified with strictly increasing functions (enumerations). Under this identification, infinite subsets of~$Z\in[\mathbb N]^\omega$ correspond to compositions~$Z\circ Y$ with~$Y\in[\mathbb N]^\omega$. The proof of Lemma~VI.6.2 in~\cite{simpson09} yields an infinite~$Z\in\mathcal M_1$ such that $\psi(x,Z\circ Y)$ is a $\Sigma^0_1$-property of $x$ and~$Y\in\mathcal M_1$ (with parameters in~$\mathcal M_1$).  Finally, the open Ramsey theorem in $\mathcal M_1\vDash\mathsf{ATR}_0$ (see Lemma~V.9.4 of~\cite{simpson09}) yields an infinite $Z'\in\mathcal M_1$ with
\begin{equation*}
\mathcal M_1\vDash\forall Y\in[\mathbb N]^\omega\,\overline f(Z\circ Z'\circ Y)=0\lor\forall Y\in[\mathbb N]^\omega\,\overline f(Z\circ Z'\circ Y)\neq 0.
\end{equation*}
So the open claim from above holds for the set $X$ that is enumerated by~$Z\circ Z'$.
\end{proof}

Let us discuss a complication that will come up in the following.

\begin{remark}\label{rmk:cof-RT-colours}
In contrast to the usual versions of Ramsey's theorem (which say that a function becomes constant on all of~$[X]^\omega$ for some~$X\in[\mathbb N]^\omega$), our cofinite \mbox{$\Delta^0_2$-}Ramsey theorem and the versions of Suzuki and Yokoyama~\cite{suzuki-yokoyama} cannot be applied iteratively. As a consequence, we do not seem able to derive the corresponding statements for more than two colours. More explicitly, to show that a $\Delta^0_2$-function $F:[\mathbb N]^\omega\to\{0,1,2\}$ is constant on some set~$[X]^\omega$, we can use a first application of the $\Delta^0_2$-Ramsey theorem to learn that $F$ takes at most two values on a set~$[Y]^\omega$. By a second application of the theorem modulo~$Y\cong\mathbb N$, we then get the desired~$X\subseteq Y$. In the case of the cofinite version, the first application only tells us that $F$ takes two values on the cofinite subsets of~$Y$, which is not enough to justify a second application.\pagebreak We could show that $\Pi^1_1\text{-}\mathsf{CA}^\Gamma_0$ also proves the cofinite $\Delta^0_2$-Ramsey theorem for each finite number of colours (combine the previous proof with the one of Lemma~3.2 from~\cite{montalban-fraisse}). This would yield a straightforward proof~of Fra\"iss\'e's conjecture in $\Pi^1_1\text{-}\mathsf{CA}^\Gamma_0$ (given the deep Theorem~\ref{thm:montalban} due to Montalb\'an). However, we wish to prove the stronger result that Fra\"iss\'e's conjecture follows from the cofinite $\Delta^0_2$-Ramsey theorem for two colours, which requires an additional argument.
\end{remark}

In view of Proposition~\ref{prop:D02Ram-Pi11star} and Theorem~\ref{thm:montalban} (the latter due to Montalb\'an), the following completes the proof of Theorem~\ref{thm:Fraisse-Pi11star} from the introduction.

\begin{theorem}[$\mathsf{ACA}_0$]
The cofinite $\Delta^0_2$-Ramsey theorem entails that the antichain with three elements is $\Delta^0_2\text{-}\mathsf{bqo}$.
\end{theorem}
\begin{proof}
Consider a function $f:[\mathbb N]^{<\omega}\to\{0,1,2\}$ that is eventually constant. In view of Definition~\ref{def:Del02-bqo}, we need to show that there is an $X\in[\mathbb N]^\omega$ with $\overline f(X)=\overline f(X^-)$, where we abbreviate $X^-=X\backslash\{\min(X)\}$. By the cofinite $\Delta^0_2$-Ramsey theorem, we find a $Z\in[\mathbb N]^\omega$ such that one of the following two cases applies. In the first case, we have $\overline f(Y)=2$ for all cofinite $Y\subseteq Z$. The open claim is then satisfied for $X=Z$ (as $Z^-\subseteq Z$ is cofinite). The rest of the proof is concerned with the remaining case, in which we have $\overline f(Y)\in\{0,1\}$ for all cofinite~$Y\subseteq Z$. We note that the following argument is very similar to Montalb\'an's proof~\cite{montalban-fraisse} that $\{0,1\}$ is $\Delta^0_2\textsf{-bqo}$ over~$\mathsf{ACA}_0$, which is based on Marcone's proof~\cite{marcone-survey-old} that $\{0,1\}$ is $\Delta^0_1\textsf{-bqo}$ over~$\mathsf{RCA}_0$.

For $s,t\in[\mathbb N]^{<\omega}$, we write $s\sqsubset t$ to indicate that $s$ is a proper initial segment of~$t$, i.\,e., that there are $Y\in[\mathbb N]^\omega$ and $m<n$ with $s=Y[m]$ and $t=Y[n]$. Given $s\in[\mathbb N]^{<\omega}$ and $Y\in[\mathbb N]^\omega$, we put $Y/s=\{n\in Y\,|\,m<n\text{ for all }m\in s\}$.

We may fix an $s\in[Z]^{<\omega}$ such that $f(s)=f(t)$ holds for all~$t\in[Z]^{<\omega}$ with~$s\sqsubset t$. Indeed, if no such $s$ did exist, we would find $s_0\sqsubset s_1\sqsubset\ldots$ with $s_i\in[Z]^{<\omega}$ and $f(s_i)\neq f(s_{i+1})$ for all~$i\in\mathbb N$. But then $f$ would fail to be eventually constant along the infinite set $\bigcup_{i\in\mathbb N}s_i$. The choice of $s$ ensures that we have
\begin{equation*}
\overline f(V)=\overline f(W)\quad\text{for }V=s\cup Z/s\text{ and }W=s\cup(Z/s)^-.
\end{equation*}
Let us declare $Y^{-0}=Y$ and $Y^{-(i+1)}=(Y^{-i})^-$, so that $Y^{-i}$ is $Y$ without its $i$ smallest elements. When~$j$ is the size of~$s$, we have $V^{-j}=Z/s$ and $W^{-j}=(Z/s)^-$. Towards a contradiction, we assume that there is no $X\in[\mathbb N]^\omega$ with $\overline f(X)=\overline f(X^-)$. Since each $V^{-i}$ is a cofinite subset of~$Z$, we get
\begin{equation*}
\overline f(V^{-i})=1-\overline f(V^{-(i+1)})\in\{0,1\}.
\end{equation*}
The same holds with $W$ at the place of~$V$. When $j$ is even, this yields
\begin{equation*}
\overline f(Z/s)=\overline f(V^{-j})=\overline f(V)=\overline f(W)=\overline f(W^{-j})=\overline f((Z/s)^-)=1-\overline f(Z/s).
\end{equation*}
A similar contradiction arises when~$j$ is odd.
\end{proof}

\bibliographystyle{amsplain}
\bibliography{Fraisse-PartialImpred}

\newcommand{\noopsort}[1]{}
\providecommand{\bysame}{\leavevmode\hbox to3em{\hrulefill}\thinspace}
\providecommand{\MR}{\relax\ifhmode\unskip\space\fi MR }
\providecommand{\MRhref}[2]{%
  \href{http://www.ams.org/mathscinet-getitem?mr=#1}{#2}
}
\providecommand{\href}[2]{#2}
\begin{thebibliography}{10}

\bibitem{TLPP-Caristi}
David Fern\'andez-Duque, Paul Shafer, Henry Towsner, and Keita Yokoyama,
  \emph{Metric fixed point theory and partial impredicativity}, Philosophical
  Transactions of the Royal Society A \textbf{381} (2023), article
  no.~20220012.

\bibitem{freund-equivalence}
Anton Freund, \emph{${\Pi}^1_1$-comprehension as a well-ordering principle},
  Advances in Mathematics \textbf{355} (2019), article no.~106767, 65~pp.

\bibitem{freund-computable}
\bysame, \emph{Computable aspects of the {B}achmann-{H}oward principle},
  Journal of Mathematical Logic \textbf{20} (2020), no.~2, article no.~2050006,
  26 pp.

\bibitem{freund-3-bqo}
\bysame, \emph{On the logical strength of the better quasi order with three
  elements}, Transactions of the American Mathematical Society \textbf{376}
  (2023), 6709--6727.

\bibitem{freund-manca-fraisse}
Anton Freund and Davide Manca, \emph{Weak well orders and {F}ra{\"i}ss{\'e}'s
  conjecture}, The Journal of Symbolic Logic (to appear),
  \url{https://dx.doi.org/10.1017\%2Fjsl.2023.70}.

\bibitem{provable-bqos}
Anton Freund, Alberto Marcone, Fedor Pakhomov, and Giovanni Sold\`a,
  \emph{Provable better quasi orders}, 2023, preprint available as
  \url{https://arxiv.org/abs/2305.01066}.

\bibitem{FR_Pi11-recursion}
Anton Freund and Michael Rathjen, \emph{Well ordering principles for iterated
  {$\Pi^1_1$}-comprehension}, Selecta Mathematica \textbf{29} (2023), article
  no.~76, 83 pp.

\bibitem{friedman-rm}
Harvey Friedman, \emph{Some systems of second order arithmetic and their use},
  Proceedings of the International Congress of Mathematicians, Vancouver 1974
  (Ralph~Duncan James, ed.), vol.~1, Canadian Mathematical Congress, 1975,
  pp.~235--242.

\bibitem{girard87}
Jean-Yves Girard, \emph{Proof theory and logical complexity, volume 1}, Studies
  in Proof Theory, Bibliopolis, Napoli, 1987.

\bibitem{hirst94}
Jeffry~L. Hirst, \emph{Reverse mathematics and ordinal exponentiation}, Annals
  of Pure and Applied Logic \textbf{66} (1994), 1--18.

\bibitem{howard-ID1}
William Howard, \emph{A system of abstract constructive ordinals}, The Journal
  of Symbolic Logic \textbf{37} (1972), no.~2, 355--372.

\bibitem{kohlenbach-book}
Ulrich Kohlenbach, \emph{Applied {P}roof {T}heory: {P}roof {I}nterpretations
  and their {U}se in {M}athematics}, Springer Monographs in Mathematics, Berlin
  and Heidelberg, 2008.

\bibitem{marcone-bad-sequence}
Alberto Marcone, \emph{On the logical strength of {N}ash-{W}illiams' theorem on
  transfinite sequences}, Logic: From Foundations to Applications (W.~Hodges,
  M.~Hyland, C.Steinhorn, and J.Truss, eds.), Oxford University Press, 1996,
  pp.~327--351.

\bibitem{marcone-survey-old}
\bysame, \emph{{WQO} and {BQO} theory in subsystems of second order
  arithmetic}, Reverse Mathematics 2001 (Stephen Simpson, ed.), Lecture Notes
  in Logic, vol.~21, Cambridge University Press, 2005, pp.~303--330.

\bibitem{marcone-montalban}
Alberto Marcone and Antonio Montalb{\'a}n, \emph{The {V}eblen functions for
  computability theorists}, The Journal of Symbolic Logic \textbf{76} (2011),
  575--602.

\bibitem{montalban-fraisse}
Antonio Montalb{\'a}n, \emph{Fra{\"i}ss{\'e}'s conjecture in
  {$\Pi^1_1$}-comprehension}, Journal of Mathematical Logic \textbf{17} (2017),
  no.~2, article no.~1750006.

\bibitem{nash-williams-trees}
Crispin~{St.\ J.\ A.} Nash-Williams, \emph{On well-quasi-ordering infinite
  trees}, Mathematical Proceedings of the Cambridge Philosophical Society
  \textbf{61} (1965), 697--720.

\bibitem{nash-williams-bqo}
\bysame, \emph{On better-quasi-ordering transfinite sequences}, Mathematical
  Proceedings of the Cambridge Philosophical Society \textbf{64} (1968),
  273--290.

\bibitem{pohlers98}
Wolfram Pohlers, \emph{Subsystems of set theory and second order number
  theory}, Handbook of Proof Theory (S.~Buss, ed.), Elsevier, 1998,
  pp.~209--335.

\bibitem{pohlers-proof-theory}
\bysame, \emph{Proof theory. {T}he first step into impredicativity}, Springer,
  Berlin, 2009.

\bibitem{rathjen-atr}
Michael Rathjen, \emph{$\omega$-models and well-ordering principles},
  Foundational Adventures: Essays in Honor of Harvey M.~Friedman (Neil Tennant,
  ed.), College Publications, 2014, pp.~179--212.

\bibitem{rathjen-model-bi}
Michael Rathjen and Pedro Francisco~Valencia Vizca\'{i}no, \emph{Well ordering
  principles and bar induction}, Gentzen's centenary: The quest for consistency
  (Reinhard Kahle and Michael Rathjen, eds.), Springer, Berlin, 2015,
  pp.~533--561.

\bibitem{rathjen-weiermann-atr}
Michael Rathjen and Andreas Weiermann, \emph{Reverse mathematics and
  well-ordering principles}, Computability in Context: Computation and Logic in
  the Real World (S.~Barry Cooper and Andrea Sorbi, eds.), Imperial College
  Press, 2011, pp.~351--370.

\bibitem{shore-comp-wos}
Richard Shore, \emph{On the strength of {F}ra{\"i}ss{\'e}'s conjecture},
  Logical Methods. In Honor of Anil Nerode's Sixtieth Birthday (John Crossley,
  Jeffrey Remmel, Richard Shore, and Moss Sweedler, eds.), Progress in Computer
  Science and Applied Logic, vol.~12, Birkh{\"a}user, Boston (MA), 1993,
  pp.~782--813.

\bibitem{simpson-borel-bqos}
Stephen Simpson, \emph{Bqo theory and {F}ra\"iss\'e's conjecture},
  \normalfont{chapter in the book} `Recursive Aspects of Descriptive Set
  Theory' \normalfont{by R.~Mansfield and G.~Weitkamp}, Oxford University
  Press, 1985, pp.~124--138.

\bibitem{simpson09}
\bysame, \emph{Subsystems of second order arithmetic}, Perspectives in Logic,
  Cambridge University Press, 2009.

\bibitem{suzuki-yokoyama}
Yudai Suzuki and Keita Yokoyama, \emph{On the ${\Pi}^1_2$-consequences of
  ${\Pi}^1_1\textsf{-CA}_0$}, 2024, \url{https://arxiv.org/abs/2402.07136}.

\bibitem{takeuti67}
Gaisi Takeuti, \emph{Consistency proofs of subsystems of classical analysis},
  Annals of Mathematics \textbf{68} (1967), 299--348.

\bibitem{thomson-thesis}
Ian~Alexander Thomson, \emph{Well-{O}rdering {P}rinciples and
  ${\Pi}^1_1$-{C}omprehension $+$ {B}ar {I}nduction}, Ph{D} thesis, University
  of Leeds, 2017.

\bibitem{thomson-rathjen-Pi-1-1}
Ian~Alexander Thomson and Michael Rathjen, \emph{Well-ordering principles,
  $\omega$-models and ${\Pi}^1_1$-comprehension}, The Legacy of Kurt
  Sch{\"u}tte (Reinhard Kahle and Michael Rathjen, eds.), Springer, Cham, 2020.

\bibitem{partial-impred}
Henry Towsner, \emph{Partial impredicativity in reverse mathematics}, The
  Journal of Symbolic Logic \textbf{2} (2013), no.~78, 459--488.

\end{thebibliography}

\end{document}